\documentclass[12pt]{amsart}

\newtheorem{thm}{Theorem}%[section]
\newtheorem{prop}[thm]{Proposition}
\newtheorem{cor}[thm]{Corollary}
\newtheorem{lemma}[thm]{Lemma}
\theoremstyle{definition}

\newtheorem{remark}[thm]{Remark}

\newcommand{\N}{\mathbb{N}}
\newcommand{\Z}{\mathbb{Z}}
\newcommand{\R}{\mathbb{R}}
\newcommand{\C}{\mathbb{C}}
\newcommand{\bigoh}{\operatorname{O}}

\newcommand{\Cinfty}{\ensuremath{C^{\infty}}}
\newcommand{\Ccinfty}{\ensuremath{C_c^{\infty}}}
\newcommand{\Dprime}{\ensuremath{C^{-\infty}}}
\newcommand{\Ltwo}{\ensuremath{L^2}}
\newcommand{\Tstar}{\ensuremath{T^*}}
\newcommand{\setof}{\,;\,}
\newcommand{\supp}{\operatorname{supp}}

\newcommand{\WF}{\operatorname{WF}}

\newcommand{\restrict}[1]{|_{#1}}
\newcommand{\Hom}{\operatorname{Hom}}

\newcommand{\I}{i} % imag Einheit
\newcommand{\Iinv}{i^{-1}} % 1/i
\newcommand{\clie}[1]{\mathfrak{#1}_{\C}}

\newcommand{\lie}[1]{\mathfrak{#1}}
\newcommand{\tr}{\operatorname{tr}}
\newcommand{\Av}{\operatorname{Av}}
\newcommand{\Ad}{\operatorname{Ad}}

\newcommand{\Char}{\operatorname{Char}}
\newcommand{\ind}{\operatorname{ind}}
\newcommand{\afsupp}{\operatorname{afsupp}}
\newcommand{\AS}{\operatorname{AS}}
\newcommand{\Weylcc}{\bar{C}}
\newcommand{\Khat}{\hat{K}}
\newcommand{\Mhat}{\hat{M}}
\newcommand{\ME}[1]{M_{#1}}

\begin{document}

\title[Asymptotic $K$-Support]{Asymptotic $K$-Support and Restrictions of Representations}
\author[S. Hansen]{S\"onke Hansen}
\author[J. Hilgert]{Joachim Hilgert}
\author[S. Keliny]{Sameh Keliny}
\address{Institut f\"ur Mathematik\\ Universit\"at Paderborn\\ 33098 Paderborn\\ Germany}
%\date{May 4, 2009}

% \begin{abstract}
% This will become a more concrete abstract.
% \end{abstract}

\maketitle

\section{Introduction}

In the late nineties T.~Kobayashi wrote a series of papers in which he established
a criterion for the discrete decomposablity of  restrictions of unitary representations
of reductive Lie groups to reductive subgroups.
A key tool in the proof of sufficiency of his criterion was the use of the theory of
hyperfunctions to study the microlocal behavior of characters of restrictions to
compact subgroups. See \cite{kobayashi:98:microlocal}.
In this paper we show how to replace this tool by microlocal analysis in the $\Cinfty$ category.

In the following $K$ denotes a connected, compact Lie group with Lie algebra $\lie{k}$.
We fix a maximal torus with Lie algebra $\lie{t}\subset\lie{k}$ and an associated positive system.
By $\Weylcc\subset\I\lie{t}^*$, $\I=\sqrt{-1}$, we denote the closure of the (dual) Weyl chamber.
We identify equivalence classes of irreducible representations with their highest weights.
Thus we write $\Khat=\Lambda\cap\Weylcc$, where $\Lambda$ denotes the weight lattice in $\I\lie{t}^*$.
We also assume an $\Ad$-invariant inner product on $\lie{k}$,
extended to an $\Ad$-invariant  hermitian inner product on the complexification $\clie{k}$.
We denote the norm of $\lambda\in\clie{k}$ by $|\lambda|$.
Using the inner products we identify $\lie{t}^*$ and $\clie{t}^*$
with subsets of $\lie{k}^*$ and of $\clie{k}^*$, respectively.

The Fourier series $u=\sum_{\lambda\in\Khat} u_\lambda$
of any square integrable function $u$ converges in $\Ltwo(K)$.
A (formal) Fourier series $\sum_{\lambda\in\Khat} u_\lambda$ converges to a distribution $u\in\Dprime(K)$
iff the $\Ltwo$ norms $\|u_\lambda\|$ of the Fourier coefficients are polynomially bounded
as functions of $\lambda\in\Khat$.
Smooth functions, $u\in\Cinfty(K)$, are characterized by the rapid decrease of their Fourier coefficients,
$\|u_\lambda\|=\bigoh(|\lambda|^{-\infty})$ as $\lambda\to\infty$.
We shall define, for every distribution $u$, a closed cone $\afsupp(u)\subset\Weylcc\setminus 0$,
the asymptotic Fourier support of $u$.
Essentially this is the smallest cone outside which the Fourier coefficients decrease rapidly.
The asymptotic Fourier support is empty for $\Cinfty$ functions.

The wavefront set is a fundamental notion in the microlocal analysis of distributions.
Given a closed cone $\Gamma\subset \Tstar K\setminus 0$
one defines the space $\Dprime_\Gamma(K)$ which consists of all $u\in\Dprime(K)$
having their wavefront sets contained in $\Gamma$, $\WF u\subset\Gamma$.
Under appropriate geometric conditions on $\Gamma$ some operations
can be extended by continuity to $\Dprime_\Gamma(K)$.
The wavefront set was used by Howe~\cite{howe:81:wf} in a related setting.

The group $K\times K$ acts on the cotangent bundle $\Tstar K$ via left and right translations.

\begin{thm}
\label{thm-wf-afsupp}
Let $u\in\Dprime(K)$.
Then
\begin{equation}
\label{eq-wf-afsupp}
(K\times K)\cdot\WF(u)= (K\times K)\cdot \Iinv\afsupp(u).
\end{equation}
The Fourier series of $u$ converges in $\Dprime_{(K\times K)\cdot\Iinv\afsupp(u)}(K)$.
\end{thm}

Kashiwara and Vergne \cite[4.5]{kashiwara/vergne:79:ktypes} proved the first
assertion in the hyperfunction setting and noticed the $\Cinfty$ analogue in a remark.
The importance of the second assertion is that it implies for subgroups satisfying
geometric assumptions that restriction commutes with Fourier series.

A representation $\pi$ of $K$ in a Hilbert space is said to be polynomially bounded
if the $K$-multiplicity $m_K(\lambda:\pi)=\dim \Hom_K(\lambda,\pi)$ of $\lambda$ in $\pi$
is polynomially bounded as a function of $\lambda\in\Khat$.
In particular, the multiplicities are finite then.
The asymptotic $K$-support of $\pi$ is a closed cone $\AS_K(\pi)\subset\Weylcc\setminus 0$
with approximates the support of $m_K(\cdot:\pi)$ as $\lambda\to\infty$.
(See \cite[(2.7.1)]{kobayashi:98:microlocal}.)

\begin{thm}[\cite{kobayashi:98:microlocal}]
\label{thm-kobayashi}
Let $M$ be a closed subgroup of $K$.
Denote its Lie algebra by $\lie{m}$, and by $\lie{m}^\perp\subset\lie{k}^*$ the space of conormals.
Let $\pi$ be a unitary representation of $K$ which is polynomially bounded and which satisfies
\begin{equation}
\label{AS-disjoint-normalbundle}
\AS_K(\pi)\cap \I\Ad^*(K)\lie{m}^\perp =\emptyset.
\end{equation}
Then the restriction $\pi\restrict{M}$ of $\pi$ to $M$ is a polynomially bounded representation of $M$.
The asymptotic $M$-support
$\AS_M(\pi\restrict{M})$ is contained in the image of $\Ad^*(K)\AS_K(\pi)$ under
the canonical projection $\I\lie{k}^*\to\I\lie{m}^*$.
\end{thm}

It is known that the restriction of an irreducible unitary representation $\pi$ of a real
reductive Lie group $G$ to a maximal compact subgroup is polynomially bounded.
For closed subgroups $G'\subset G$ which are stable under the Cartan involution a criterion
on $G'$-admissability of $\pi\restrict{G'}$ is given in \cite[Theorem~2.9]{kobayashi:98:microlocal}.
Theorem~\ref{thm-kobayashi} contains all the micro-local information needed to rewrite the proof
of \cite[Theorem~2.9]{kobayashi:98:microlocal} without having to invoke the theory of hyperfunctions.
Thus we offer an alternative approach to Kobayashi's theorem
for readers without a strong background in hyperfunction theory.

The proof of Theorem~\ref{thm-kobayashi} is centered around the notion of the $K$-character
$\sum_{\lambda\in\Khat} m_K(\lambda:\pi) \tr\lambda$ of $\pi$.
The asumptions imply that this series converges in $\Dprime(K)$,
and that the $K$-character posesses a restriction to $M$ which turns out
to be the $M$-character of $\pi\restrict{M}$.
Theorem~\ref{thm-wf-afsupp} is used to prove this.
The continuity statement given in Theorem~\ref{thm-wf-afsupp}
simplifies the proof Theorem~\ref{thm-kobayashi} when compared with the original argument.

The paper is organized as follows.
In Section~\ref{s-afsupp} we recall the expansion in eigenfunctions of a positive elliptic
operator and its application to Fourier series on $K$.
The asymptotic Fourier support is defined in this section.
In Section~\ref{s-wfcd} we study, for central distributions,
wavefronts sets and the convergence of Fourier series.
The theorems are proved in Sections~\ref{s-thm1} and \ref{s-thm2}.

This research grew out of the dissertation of the third author.
The work was supported by the DFG via the international research training
group ``Geometry and Analysis of Symmetries''.

\section{Asymptotic Fourier support}
\label{s-afsupp}

The space $\Dprime(K)$ of distributions on $K$ is, by definition, the dual space of $\Cinfty(K)$.
Functions are identified with distributions, $\Ltwo(K)\subset\Dprime(K)$,
using the normalized Haar measure $dk$ on $K$.
The $\Ltwo$ scalar product $(\cdot|\cdot)$ extends to an anti-duality between $\Dprime(K)$ and $\Cinfty(K)$.
We recall how the theory of Fourier series of distributions and of smooth functions
follows from results on eigenfunction expansions of elliptic selfadjoint differential operators.

The Sobolev space $H^m(K)$ consists of all distributions which are mapped into
$\Ltwo(K)$ by differential operators with order $\leq m$.
We assume differential operators to be linear with $\Cinfty$ coefficients.
$H^m(K)$ is equipped with a norm making it a Banach space.
Let $A$ be a second order, elliptic differential operator.
Regard $A$ as an unbounded operator on $\Ltwo(K)$ with domain $D(A)=H^2(K)$.
Its Hilbert space adjoint $A^*$ has, by elliptic regularity theory, the domain $D(A^*)=H^2(K)$.
Assume, in addition, that $(Au|u)>0$ if $0\neq u\in D(A)$.
Then $A$ is positive selfadjoint.
The eigenfunctions of $A$ are in $\Cinfty(K)$.

\begin{prop}[{\cite[\S 10]{seeley:65:integrodiff}}]
\label{seeleyThm}
Let $A$ be a positive selfadjoint second order elliptic differential operator on $K$.
Let $Au=\sum_j \mu_j^2(u|e_j)e_j$ denote its spectral resolution where $(e_j)\subset\Ltwo(K)$
is an orthonormal basis of eigenfunctions and $0<\mu_j\uparrow\infty$ the corresponding
sequence of eigenvalues of $\sqrt{A}$.
A series $\sum_j\alpha_j e_j$ converges in $\Cinfty(K)$ iff
$\sum_j \mu_j^{2N}|\alpha_j|^2<\infty$ for all $N\in\N$.
It converges in $\Dprime(K)$ iff $\sum_j \mu_j^{-2N}|\alpha_j|^2<\infty$ for some $N\in\N$.
The coefficients are $\alpha_j=(u|e_j)$ if $u\in\Dprime(K)$ denotes the sum of the series.
\end{prop}

\begin{proof}
The domain $D(A^k)$ of the $k$-th power of $A$ consists of all $u$ such that
$\sum_{j} \mu_j^{2k} |(u|e_j)|^2<\infty$.
We equip $D(A^k)$ with the corresponding norm.
The norm is equivalent with the graph norm.
Hence $D(A^k)$ is a Banach space.
Obviously, $H^{2k}(K)\subset D(A^k)$.
By elliptic regularity we have equality, $D(A^k)=H^{2k}(K)$.
This holds also topologically because of Banach's theorem.
By the Sobolev lemma, $\Cinfty(K)=\cap_{k} H^{2k}(K)$ as a projective limit.
Hence the norms on $D(A^k)$ define the Fr\'echet space topology of $\Cinfty(K)$.
The asserted convergence criterion for $\Cinfty(K)$ follows from this.
Using duality between weighted $\ell^2$ sequence spaces
we obtain the convergence criterion for $\Dprime(K)$.
Finally, the formula for the coefficients follows from
the (separate) continuity of the anti-duality bracket.
\end{proof}

The $\ell^2$ estimates in Proposition~\ref{seeleyThm}
can be replaced by supremum estimates because
$\sum_j \mu_j^{-k}<\infty$ for some $k\in\N$.
The latter property holds because $A^{-k/2}$ is of trace class if $k>\dim K$.

Denote by $d_\lambda$, $\chi_\lambda=\tr\lambda$, and $\ME{\lambda}\subset\Ltwo(K)$
the dimension, the character and the space of matrix coefficients of $\lambda\in\Khat$.
The convolution with $d_\lambda\chi_\lambda$ is the orthoprojector from $\Ltwo(K)$ onto $\ME{\lambda}$.
If $u\in\Ltwo(K)$, then its Fourier series $\sum_{\lambda\in\Khat} u_\lambda$,
$u_\lambda=d_\lambda u*\chi_\lambda$, converges to $u$ in $\Ltwo(K)$ by the Peter-Weyl theorem.
The (formal) Fourier series of a distribution $u\in\Dprime(K)$ is defined by the same formula
using the convolution of a distribution with a $\Cinfty$ function,
i.e., $(u*\psi)(x)=\int_K u(y)\psi(y^{-1}x)\,dy$ for $\psi\in\Cinfty(K)$ with
the integral representing the duality bracket.
Observe that $\chi_\lambda, u*\chi_\lambda\in\ME{\lambda}\subset\Cinfty(K)$.
In general, we call a series $\sum_{\lambda\in\Khat} u_\lambda$ with $u_\lambda\in\ME{\lambda}$
a Fourier series with coefficients $u_\lambda$.

We use left translation, $L_x(k)= xk$, to trivialize
the tangent bundle $TK=K\times\lie{k}$ and the cotangent bundle $\Tstar K=K\times\lie{k}^*$.
Under this identication left translation is the identity on the second components.
Right translation $R_x(y)=yx$ acts, on the second components, as the adjoint action,
$dR_{x^{-1}}:X\mapsto\Ad(x)X$, and as the co-adjoint action, ${}^t dR_{x^{-1}}:\xi\mapsto\Ad^*(x)\xi$.
Bi-invariant subsets of $\Tstar K$ are of the form $K\times \Ad^*(K)S$ for some $S\subset\lie{k}$.
Then formula~\eqref{eq-wf-afsupp} reads
\begin{equation*}
(K\times K)\cdot\WF(u)= K\times \Iinv\Ad^*(K)\afsupp(u).
\end{equation*}

Elements $X\in U(\clie{k})$ of the universal enveloping algebra act as left invariant
differential operators $\widetilde{X}$ on $\Dprime(K)$.
The principal symbol of the first order differential operator $\widetilde{X}$
associated with $X\in\clie{k}$ is $\sigma_1(\widetilde{X})(x,\xi)=\langle\xi,X\rangle$.
Denote the $\Ad$-invariant hermitian inner product on $\clie{k}$ by $Q$.
We assume that $Q$ equals the negative Killing form on $[\lie{k},\lie{k}]$
and that the center of the Lie algebra is orthogonal to $[\lie{k},\lie{k}]$.
Choose, consistent with this orthogonal decomposition, an orthonormal basis $\{X_j\}$ of $\lie{k}$.
Define the second order differential operator $A=1-\sum_j \widetilde{X}_j^2$.
The principal symbol of $A$ is $\sigma_2(A)(x,\xi)=Q^*(\xi)$, where $Q^*$ is the dual form of $Q$.
Hence $A$ is elliptic.
It follows from the left invariance of $\widetilde{X}$  and the invariance
of Haar measure that $\int_K\widetilde{X}v(y)\,dy=0$ for all $v\in\Cinfty(K)$, $X\in\clie{k}$.
Therefore, $A$ is positive selfadjoint with domain $H^2(K)$.
Furthermore, $A$ is bi-invariant.
Therefore, each $\ME{\lambda}$, $\lambda\in\Khat$, is contained in an eigenspace of $A$
with eigenvalue $\mu=\mu(\lambda)$.
There exists a constant $C>0$ such that
\begin{equation*}
1+|\lambda+\rho|^2-|\rho|^2\leq \mu\leq C(1+|\lambda|)^2
\quad\text{for all $\lambda\in\Khat$.}
\end{equation*}
Here $\rho$ is the half sum of positive roots.
The left inequality holds because $A-1$ is the sum of a non-negative operator $B$ and the Casimir operator.
It is well-known that the Casimir operator contains $\ME{\lambda}$ in its eigenspace
with eigenvalue $|\lambda+\rho|^2-|\rho|^2$.
Since $B$ is a sum of $-\widetilde{X}^2$, $X\in\lie{t}$, the right inequality follows from
$(-\widetilde{X}^2u|u)=\|\widetilde{X}u\|^2=\|\langle\lambda,X\rangle u\|^2$
which holds for any highest weight vector $u\in\ME{\lambda}$.

Summarizing we have the following.

\begin{cor}
\label{conv-F-series}
A Fourier series $\sum_{\lambda\in\Khat} u_\lambda$ converges
in $\Cinfty(K)$, resp.~in $\Dprime(K)$, iff
\begin{equation*}
\sum_{\lambda\in\Khat} (1+|\lambda|)^{2N} \|u_\lambda\|^2 <\infty
\end{equation*}
for all, resp.~for some, $N\in\Z$.
If $u\in\Dprime(K)$, then its Fourier series $\sum_{\lambda\in\Khat}u_\lambda$,
$u_\lambda=d_\lambda u*\chi_\lambda$, converges to $u$ in $\Dprime(K)$.
\end{cor}

Smoothness properties of a distribution correspond to decaying properties of its Fourier coefficients.
We define an approximating cone to the directions of those $\lambda\in\Khat\subset\Weylcc$
such that the Fourier coefficients $u_\lambda$ do not decay rapidly as $\lambda\to\infty$.
A subset of a (finite dimensional) real vector space $V$ (or of a vector bundle)
is called conic or a \emph{cone} iff it is invariant under multiplication with positive reals.

Let $u\in\Dprime(K)$ and $\sum_{\lambda\in \Khat} u_\lambda$ its Fourier series.
The \emph{asymptotic Fourier support} of $u$ is the closed cone
$\afsupp(u)\subset\Weylcc\setminus 0$ which is defined as follows.
A point $\mu\in\Weylcc\setminus 0$ is in the complement of $\afsupp(u)$
iff there is a conic neighbourhood $S\subset\Weylcc\setminus 0$ of $\mu$ such that
\begin{equation*}
\sum_{\lambda\in S\cap\Khat} |\lambda|^{2N} \|u_\lambda\|^2 <\infty,
\quad \forall N\in\N.
\end{equation*}
By Corollary~\ref{conv-F-series}, $u\in\Cinfty(K)$ iff $\afsupp(u)=\emptyset$.
More generally, if $S\subset\Weylcc\setminus 0$ is a closed cone which is disjoint from $\afsupp(u)$,
then the Fourier series $\sum_{\lambda\in S\cap\Khat}u_\lambda$ converges in $\Cinfty(K)$.

\begin{remark}
Instead of working with $\ell^2$-estimates we can work with supremum estimates such as
$\sup_{\lambda\in S\cap\Khat} |\lambda|^{N} \|u_\lambda\| <\infty$.
This follows from the observation made after the proof of Proposition~\ref{seeleyThm}.
\end{remark}

With a subset $S\subset V$ one associates the closed cone $S_\infty\subset V\setminus 0$ as follows.
A point is in the complement of $S_\infty$ if it has a conic neighbourhood
which intersects $S$ in a relatively compact set.
Equivalently, $v\in S_\infty$ iff there exist sequences $(v_j)\subset S$ and $\varepsilon_j\downarrow 0$
such that $\lim_j \varepsilon_j v_j =v$.
The cone $S_\infty$ approximates $S$ at infinity.

The $K$-support $\supp_K(\pi)$ of a representation $\pi$ of $K$ in a Hilbert space is the set
of all $\lambda\in \Khat\subset\Weylcc$ such that $\lambda$ occurs in $\pi$, i.e., $m_K(\lambda:\pi)>0$.
The set $\AS_K(\pi)=\supp_K(\pi)_\infty\subset\Weylcc\setminus 0$ is the \emph{asymptotic $K$-support} of $\pi$.

\section{Wavefront convergence of central Fourier series}
\label{s-wfcd}

The definition of the wavefront set of a distribution is based on the calculus of pseudodifferential operators.
We collect, in our context, some definitions and results, refering to
\cite[Section~2.5]{hormander:71:FIO-1}, \cite{duistermaat:73:fio},
and \cite[Section~18.1]{hormander:85:the-analysis-3} for details.

With every pseudodifferential operator $A\in\Psi^m(K)$ one associates its
set $\Char A\subset \Tstar K\setminus 0$ of characteristic points.
A point is non-characteristic if there is a symbol $b\in S^{-m}$ such that $ab-1\in S^{-1}$
in a conic neighbourhood of that point.
Here $a\in S^m(\Tstar K)$ is, modulo $S^{m-1}(\Tstar K)$, a principal symbol of $A$.
The operator is said to be elliptic at a non-characteristic point.
An operator $A:\Cinfty(K)\to\Dprime(K)$ is a pseudodifferential operator iff its
Schwartz kernel $K_A\in\Dprime(K\times K)$ is a conormal distribution respect to the diagonal.
More explicitly, $A\in\Psi^m(K)$ iff the singular support of $K_A$ is contained in the diagonal
and $K$ can be covered with open sets $U\subset K$
such that the kernel is given by an oscillatory integral
\begin{equation}
\label{A-in-local-coord}
K_A(y',y)=\int e^{\I\varphi(y',\eta)-\I\varphi(y,\eta)} a(y',y,\eta) \,d\eta,
\end{equation}
$y',y\in U$.
The phase function $\varphi\in\Cinfty(U\times\lie{k})$ is real-valued, linear in the second variable,
and nondegenerate, i.e., $\det\varphi_{y\eta}''\neq 0$.
The amplitude $a$ belongs to the symbol space $S^m(U\times U\times\lie{k}^*)$.
$A$ is elliptic at $\xi=\varphi_x'(x,\zeta)\in\Tstar_x K\setminus 0$, $x\in U$,
iff there is a neighbourhood $U_0\subset U$ of $x$, a conic neighbourhood $V$ of $\zeta$,
and $C>0$ such that
$|a(y,y,\eta)|\geq |\eta|^m/C$ for $y\in U_0$, $\eta\in V$, $|\eta|>C$.

Let $u\in\Dprime(K)$. The wavefront set $\WF u\subset \Tstar K\setminus 0$
equals $\cap \Char A$, where the intersection is taken over all
pseudodifferential operators $A$ which satisfy $Au\in\Cinfty(K)$.
Let $\Gamma\subset \Tstar K\setminus 0$ be a closed cone.
The space $\Dprime_\Gamma(K)$ of distributions on $K$ which have their
wavefront sets contained in $\Gamma$ is equipped with a locally convex topology.
It contains $\Cinfty(K)$ as a sequentially dense subspace.
Convergence of a sequence, $u_j\to u$ in $\Dprime_\Gamma(K)$, is equivalent to
$u_j\to u$ (weakly) in $\Dprime(K)$  and the existence, for every
$(x,\xi)\in(\Tstar K\setminus 0)\setminus\Gamma$,
of a pseudodifferential operator $A\in\Psi^m(K)$ such that
$(x,\xi)\not\in\Char A$, and $Au_j\to Au$ in $\Cinfty(K)$.
If $u_j\to u$ in $\Dprime_\Gamma(K)$, then $Au_j\to Au$ in $\Cinfty(K)$
for every $A\in\Psi^m(K)$ which satisfies $\WF(A)\cap\Gamma=\emptyset$.
Here $\WF(A)$ is the smallest conic subset of $\Tstar K\setminus 0$
such that $A$ is of order $-\infty$ in the complement.
(See the remark following Theorem 18.1.28 of \cite{hormander:85:the-analysis-3}.)

Let $K$ act on $\Cinfty(K)$ via the right regular representation, $R_x f(y)=f(yx)$.
The corresponding action of the Lie algebra $\clie{k}$ is by left invariant vector fields,
$dR_e(X)f=\widetilde{X}f$.

The following lemma should be compared with \cite[3.1]{kashiwara/vergne:79:ktypes}.

\begin{lemma}
\label{highestweight}
Let $\sum_{\lambda\in \Khat} u_\lambda$ be a Fourier series which converges in $\Dprime(K)$.
Assume that each $u_\lambda$ is a highest weight vector for the right regular representation
acting irreducibly on a subspace of $\ME{\lambda}$.
Let $S$ be a closed cone $\subset\Weylcc\setminus 0$.
Then $\sum_{\lambda\in S\cap\Khat} u_\lambda$ converges in $\Dprime_{K\times\Iinv S}(K)$.
\end{lemma}
\begin{proof}
The differential equations
$\widetilde{X}u_\lambda=0$ and $\widetilde{X}u_\lambda=\langle\lambda,X\rangle u_\lambda$
hold for $X\in\lie{n}$ and $X\in\lie{t}$, respectively.
Here $\lie{n}\subset\clie{k}$ denotes the sum of positive root spaces.

Let $(x,\xi)\in K\times\lie{k}^*\setminus 0$, $\xi\not\in\Iinv S$.
It suffices to find a pseudodifferential operator $A$, elliptic at $(x,\xi)$, such that
the series $\sum_{\lambda\in S\cap\Khat} Au_\lambda$ converges in $\Cinfty(K)$.
If $\xi\not\in \lie{t}^*$, then there exists $X\in \lie{n}$ with $\langle\xi,X\rangle\neq 0$;
the first order differential operator $A=\widetilde{X}$ has the desired properties.
Now assume $\xi\in \lie{t}^*$.
Then the cone $S-\R_+\I\xi$ is a closed subset of $\I\lie{t}^*\setminus 0$.
It follows by a simple compactness argument that $|\lambda|+|\xi|\leq C|\lambda-\I\xi|$
with a constant $C>0$ independent of $\lambda\in S$.
Assume that $S$ is convex.
Choose $X\in\lie{t}$ which strictly separates the disjoint convex cones $-\R_+\xi$ and $\I S$.
We infer that there exists $c>0$ such that
\begin{equation}
\label{etalambda}
\big|\langle\lambda -\I \eta,X\rangle\big|> c(|\lambda|+|\eta|)
\quad \text{for all $\lambda\in S$, $\eta\in\Gamma$,}
\end{equation}
where $\Gamma=\R_+\xi$.
By continuity \eqref{etalambda} also holds in a conic neighbourhood $\Gamma\subset\lie{k}^*\setminus 0$ of $\xi$.

%We construct local coordinates in which $\widetilde{X}$ is a differential operator with constant coefficients.
Let $U\subset K$ be an open neighbourhood of $x$ and $H\subset U$ a hypersurface
containing $x$ such that the following holds.
The real vector field $\widetilde{X}$ is transversal to $H$ and every maximally extended
integral curve of $\widetilde{X}$ in $U$ hits $H$ in a unique point.
Furthermore,
$y\mapsto \exp^{-1}(x^{-1}y)$ maps $U$ diffeomorphically onto an open neighbourhood of the origin in $\lie{k}$.
Using the method of characteristics  we solve, for every $\eta\in\lie{k}^*$,
the initial value problem
\begin{equation*}
\widetilde{X}\varphi(\cdot,\eta)=\langle \eta,X\rangle \;\text{in $U$},
\quad \varphi(y,\eta)=\langle \eta, \exp^{-1}(x^{-1}y)\rangle\;\text{at $y\in H$.}
\end{equation*}
The solution $\varphi\in\Cinfty(U\times\lie{k}^*)$ is linear in the second variable
and $\varphi_x'(x,\eta)=\eta$ holds in $\Tstar_x K=\lie{k}^*$ for all $\eta$.
In particular, $\varphi_{y\eta}''$ is nondegenerate at $y=x$.
We have
$\widetilde{X}e^{-\I\varphi(\cdot,\eta)}u_\lambda
   =\langle\lambda -\I \eta,X\rangle e^{-\I\varphi(\cdot,\eta)}u_\lambda$.
Since $\widetilde{X}$ is left invariant, $\int_K\widetilde{X}v(y)\,dy=0$ holds for all $v\in\Cinfty(K)$.
Therefore we can perform partial integration as follows,
\begin{equation*}
\langle\lambda -\I \eta,X\rangle \int_K e^{-\I\varphi(y,\eta)}u_\lambda(y) \chi(y)\,dy
  = - \int_K e^{-\I\varphi(y,\eta)}u_\lambda(y) \widetilde{X}\chi(y)\,dy
\end{equation*}
if $\chi\in\Ccinfty(U)$.
Iterating $N$ times and estimating the integral on the right using the Cauchy-Schwarz inequality we obtain
\begin{equation*}
\big|\langle\lambda -\I \eta,X\rangle^N \int_K e^{-\I\varphi(y,\eta)}u_\lambda(y) \chi(y)\,dy\big|
   \leq C_N \|u_\lambda\|
\end{equation*}
with a constant $C_N>0$ independent of $\lambda\in S\cap\Khat$ and $\eta\in\Gamma$.
In view of \eqref{etalambda} we get
\begin{equation*}
\sup_{\eta\in\Gamma} |\eta|^{N}\big|\int_K e^{-\I\varphi(y,\eta)} \chi(y) u_\lambda(y) \,dy\big|
   \leq C_{2N}c^{-2N} |\lambda|^{-N} \|u_\lambda\|,
\end{equation*}
for all $\lambda\in S\cap\Khat$ and $N\in\N$.
Since the $\Ltwo$ norms of the Fourier coefficients are polynomially bounded we obtain,
for every $\chi\in\Ccinfty(U)$,
\begin{equation}
\label{hw-estimate-of-sum}
\sum_{\lambda\in S\cap\Khat}
 \sup_{\eta\in\Gamma} |\eta|^N \big|\int_K e^{-\I\varphi(y,\eta)} \chi(y) u_\lambda(y) \,dy\big| <\infty,
 \quad N\in\N.
\end{equation}
We can assume that, making $U$ and $\Gamma$ smaller if necessary,
$\det\varphi_{y\eta}''\neq 0$ in $U\times \lie{k}^*$,
and $\varphi_y'(U\times\Gamma)\cap\Iinv S=\emptyset$.
Fix $\chi\in\Ccinfty(U)$ with $\chi(x)=1$.
Choose a symbol $b\in S^0(\lie{k}^*)$ with $\supp b\subset\Gamma$
and $b=1$ in a conic neighbourhood of $\xi$ minus a compact set.
Define the pseudodifferential operator $A\in\Psi^0(K)$ with kernel $K_A$ supported in $U\times U$
and given by \eqref{A-in-local-coord} with amplitude $a(y',y,\eta)=\chi(y')b(\eta)\chi(y)$.
It follows from \eqref{hw-estimate-of-sum} that
$\sum_{\lambda\in S\cap\Khat} Au_\lambda$ converges in $\Cinfty(K)$.
Furthermore, $A$ is elliptic at $(x,\xi)$.
Hence we have proved the assertion under the additional assumption that $S$ is convex.
To remove this assumption observe that
$S$ can be covered by finitely many closed convex cones each not containing $\I\xi$.
Decompose the Fourier series correspondingly.
\end{proof}

Pullback and pushforward of distributions is well-defined and continuous under
assumptions on the wavefront sets.
With any $\Cinfty$ map $f:X\to Y$ map between smooth manifolds one associates its canonical relation
\begin{equation*}
C_f=\{(y,\eta;x,\xi)\setof y=f(x), \xi={}^t{f'(x)}\eta\}\subset \Tstar Y\times\Tstar X.
\end{equation*}
For a closed cone $\Gamma\subset\Tstar Y\setminus 0$ define its pullback
cone $f^*\Gamma=C_f^{-1}\circ\Gamma\subset\Tstar X$.
If $f^*\Gamma$ does not intersect the zero section,
then the pullback $f^*u=u\circ f$ extends from $\Cinfty(X)$ to a
(sequentially) continuous pullback operator $f^*:\Dprime_{\Gamma}(Y)\to \Dprime_{f^*\Gamma}(X)$.
If $f$ is a proper map, then the pushforward operator $f_*:\Dprime(X)\to\Dprime(Y)$ is defined by duality.
If, in addition, $f$ is a submersion and $\Gamma\subset\Tstar X\setminus 0$ is a closed
cone, then $f_*\Gamma:=C_f\circ\Gamma\subset\Tstar Y\setminus 0$ and
the pushforward restricts to a (sequentially) continuous map
$f_*:\Dprime_{\Gamma}(X)\to \Dprime_{f_*\Gamma}(Y)$.

An important example of a pullback operator is the restriction to a submanifold $M\subset K$.
It is defined on distributions having wavefront sets disjoint from the conormal bundle of $M$.
The pushforward by a projection $(x,y)\mapsto x$ is integration along fibers.

\begin{lemma}
\label{averaging}
Let $\Gamma\subset\Tstar K\setminus 0$ be a $K\times K$-invariant closed cone.
Taking the average $\Av f(x)=\int_K f(yxy^{-1})\,dy$ of a function $f$
extends uniquely from $\Cinfty(K)$ to an operator
$\Av:\Dprime_\Gamma(K)\to\Dprime_{\Gamma}(K)$.
\end{lemma}
\begin{proof}
Define $g:K\times K\to K$, $g(x,y)=yxy^{-1}$, and $p:K\times K\to K$, $p(x,y)=x$.
Then $\Av=p_* g^*$ on $\Cinfty(K)$.
By assumption $\Gamma=K\times S$ where $S\subset\lie{k}^*\setminus 0$ is an $\Ad^*$-invariant closed cone.
A computation shows that $((yxy^{-1},\zeta),(x,y,\xi,\eta))\in C_g$ iff
$\xi=\Ad^*(y^{-1}) \zeta$, and $\eta=\Ad^*(xy^{-1})\zeta-\Ad^*(y^{-1})\zeta$.
Clearly, $g^*\Gamma$ does not intersect the zero section.
Hence the pullback operator $g^*$ is defined.
Composing $C_g^{-1}$ with the relation $C_p$ leads to $\eta=0$ and
$p_* g^*\Gamma\subset K\times \Ad^*(K)T$.
The assertion follows from this.
\end{proof}

A distribution on $K$ is called central if it is invariant under conjugation.
The Fourier coefficients of central distributions are multiples of characters.

\begin{prop}[{\cite[4.5]{kashiwara/vergne:79:ktypes}}]
\label{wf-central}
Let $S$ be a closed cone $\subset\Weylcc\setminus 0$.
Let $u\in\Dprime(K)$ central with $(u|\chi_\lambda)=0$ if $\lambda\not\in S$.
Then the Fourier series of $u$ converges in $\Dprime_{K\times \Iinv \Ad^*(K)S}(K)$.
\end{prop}
\begin{proof}
For each $\lambda\in\Khat$ we choose a highest weight vector $w_\lambda\in\ME{\lambda}$
of a irreducible subrepresentation $\subset\ME{\lambda}$ of the right regular representation.
We may view $w_\lambda$ as a matrix coefficient of the form
$w_\lambda(x)=(R_x v| v)$ with $v\in\ME{\lambda}$, $\|v\|=1$.
Then $w_\lambda(e)=1$, and
$\|w_\lambda\|\leq \sup_K|w_\lambda|\leq 1$ by the Cauchy-Schwarz inequality.
The central function $\Av w_\lambda$ is a multiple of $\chi_\lambda$.
Comparing values at $e$ we get $\Av w_\lambda=d_{\lambda}^{-1}\chi_\lambda$.

The dimension $d_\lambda$ and, in view of Corollary~\ref{conv-F-series},
the Fourier coefficients $(u|\chi_\lambda)$ of $u$ grow at most polynomially in $\lambda$.
Hence $w=\sum_{\lambda\in S\cap\Khat} d_{\lambda}(u|\chi_\lambda)w_\lambda$ converges in $\Dprime(K)$.
By Lemma~\ref{highestweight} the series converges in $\Dprime_{K\times\Iinv S}(K)$.
The assertion follows from Lemma~\ref{averaging} since $u=\Av w$.
\end{proof}

\section{Proof of Theorem~\ref{thm-wf-afsupp}}
\label{s-thm1}

Every $v\in\Dprime(K)$ defines a convolution operator $\Cinfty(K)\to\Cinfty(K)$, $w\mapsto v*w$.
This is a continuous linear map which commutes with right translations.
Conversely, every such map is given by convolution with a unique element $v\in\Dprime(K)$.
Composition of maps defines the convolution $u*v\in\Dprime(K)$ of distributions $u,v\in\Dprime(K)$
by $(u*v)*w=u*(v*w)$, $w\in\Cinfty(K)$.
We have the formula $u*v = p_* f^*(u\otimes v)$, where $f$ is the diffeomorphism
$f:K\times K\to K\times K$, $f(x,y)=(y,y^{-1}x)$, and $p$ the projection $p:K\times K\to K$, $p(x,y)=x$.
The formula is evident for smooth functions and extends to distributions by separate sequential continuity.
The composition $C_p\circ C_f^{-1}$ of the canonical relations consists of all
\begin{equation*}
((x,\xi),(y,y^{-1}x,\Ad^*(x^{-1}y)\xi,\xi))\in \Tstar K\times \Tstar (K\times K).
\end{equation*}
The wavefront of a tensor product satisfies
\begin{equation*}
\WF(u\otimes v)\subset (\WF u\times \WF v)\cup (0\times\WF v)\cup (\WF u\times 0).
\end{equation*}
Moreover, as a bilinear map the tensor product satisfies corresponding separate continuity properties.
It follows that, for any two cones $S_1$ and $S_2$ in $\lie{k}^*\setminus 0$,
the convolution $(u_1,u_2)\mapsto u_1*u_2$ defines a separately sequentially continuous bilinear map
\begin{equation}
\label{WF-in-convolution}
*:\Dprime_{K\times S_1}(K)\times \Dprime_{K\times S_2}(K)\to \Dprime_{K\times (S_1\cap \Ad^*(K)S_2)}(K).
\end{equation}

Convolution with the Dirac distribution
$\delta=\sum_{\lambda\in \Khat} d_\lambda \chi_\lambda\in\Dprime(K)$
is the identity, $\delta*u=u$.
In the proof of the theorem we need
$\delta_S= \sum_{\lambda\in S\cap\Khat} d_\lambda \chi_\lambda\in\Dprime(K)$ where $S\subset\Weylcc\setminus 0$.
If $S$ is a closed cone, then it follows from Proposition~\ref{wf-central} that the series
also converges to $\delta_S$ in $\Dprime_{K\times \Iinv\Ad^*(K)S}(K)$.

Now, turning to the proof of the theorem, let $u\in\Dprime(K)$.
Assume that $S\subset\Weylcc\setminus 0$ is a closed cone which contains $\afsupp(u)$ in its interior.
Then the series of $\delta_{\Weylcc\setminus S} *u$ converges in $\Cinfty(K)$.
Using \eqref{WF-in-convolution} with $S_1=\Iinv \Ad^*(K)S$
we deduce from the above that the Fourier series of
$\delta_{S}*u$ converges in $\Dprime_{K\times \Iinv\Ad^*(K)S}(K)$.
It follows that the Fourier series of $u=\delta_{S}*u+\delta_{\Weylcc\setminus S}*u$
converges in this space, too.
In particular, we have
$(K\times K)\cdot\WF(u)\subset K\times \Iinv\Ad^*(K)S$.
This implies that the left-hand side in \eqref{eq-wf-afsupp} is contained in the right-hand side.

To prove the opposite inclusion let $S$ a closed cone
$\subset\Weylcc$ such that $\WF(u)\cap
(K\times\Iinv\Ad^*(K)S)=\emptyset$. We apply
\eqref{WF-in-convolution} to $\delta_S*u$ and deduce that the
Fourier series $\sum_{\lambda\in S\cap\Khat}u_\lambda$ converges in
$\Cinfty(K)$. This implies that $S$ is disjoint from the asymptotic
Fourier support of $u$. Since the closure of a Weyl chamber is a
fundamental domain for the coadjoint action on $\lie{k}^*$, this
implies $K\times \Iinv\afsupp(u)\subset (K\times K)\cdot \WF u$.

\section{Proof of Theorem~\ref{thm-kobayashi}}
\label{s-thm2}

The polynomial boundedness of $\pi$ implies the
finiteness of the multiplicities $m_K(\lambda:\pi)$ and the convergence of its $K$-character
\begin{equation} \label{def-K-char}
\Theta^K_\pi:= \sum_{\lambda\in\Khat} m_K(\lambda:\pi) \tr\lambda
\quad\text{in $\Dprime(K)$.}
\end{equation}
The support of $\Theta^K_\pi$ equals $\supp_K(\pi)$.
We have
\begin{equation*}
\AS_K(\pi)=\supp_K(\pi)_\infty=\afsupp(\Theta^K_\pi).
\end{equation*}
The second equality holds because
the $\Ltwo$-norm of each non-zero summand in \eqref{def-K-char} is $\geq 1$.
>From Theorem~\ref{thm-wf-afsupp} it follows that \eqref{def-K-char} converges
in $\Dprime_{\Gamma}(K)$ where $\Gamma=K\times \Iinv\Ad^*(K)\AS_K(\pi)$.
Assumption~\eqref{AS-disjoint-normalbundle} implies that
the conormal bundle of $M$, which is a subset of $K\times \lie{m}^\perp$,
is disjoint from $\Gamma$.
Hence the restriction
\begin{equation} \label{K-char-to-M}
\Theta^K_\pi\restrict{M} = \sum_{\lambda\in\Khat} m_K(\lambda:\pi) \tr\lambda\restrict{M}
\quad\text{converges in $\Dprime(M)$,}
\end{equation}
and $\WF(\Theta^K_\pi\restrict{M})$ is contained in $M\times \Iinv A\subset\Tstar M\setminus 0$,
where $A\subset\I\lie{m}$ denotes the image of $\Ad^*(K)\AS_K(\pi)$
under the projection $\I\lie{k}^*\to\I\lie{m}^*$.
>From Theorem~\ref{thm-wf-afsupp} it follows that
\begin{equation*}
\afsupp_M(\Theta^K_\pi\restrict{M}) \subset A=\Ad^*(M)A.
\end{equation*}
Consider, in $\Dprime(M)$, the Fourier series $\Theta^K_\pi\restrict{M}=\sum_{\mu\in\Mhat}c_\mu\tr\mu$.
By Corollary~\ref{conv-F-series} the map $\mu\mapsto c_\mu$ is polynomially bounded.
We prove that
\begin{equation*}
c_\mu=m_M(\mu:\pi\restrict{M}) \quad\text{for all $\mu\in\Mhat$.}
\end{equation*}
The assertions of the theorem will follow from this.
Moreover, it says that
$\Theta^M_{\pi\restrict{M}} =\Theta^K_\pi\restrict{M}$.

Let $\mu\in\Mhat$. Fix a representation space $H_\mu$.
Let $\rho=\ind_M^K(\mu)$ denote the unitary representation of $K$ induced by $\mu$.
We view the representation space $H_\rho$ of $\rho$ as the subspace of $\Ltwo(K,H_\mu)$
defined by $f(xm)=\mu(m^{-1})f(x)$, $m\in M$, almost every $x\in K$.
Then $f\in\Ltwo(K,H_\mu)$ belongs to $H_\rho$ only if it satisfies, in the sense of distributions, the
first order system of differential equations $\widetilde{Y}f+\mu_*(Y)f=0$, $Y\in\lie{m}$.
Here $\mu_*$ is the Lie algebra representation induced by $\mu$.
The characteristic variety of $\widetilde{Y}+\mu_*(Y)$ is contained in $K\times Y^\perp$.
Hence $\WF(f)\subset K\times\Ad^*(K)\lie{m}^\perp$ if $f\in H_\rho$.
Theorem~\ref{thm-wf-afsupp}, generalized to vector valued distributions, implies that
$\afsupp_K(f)\subset \I\Ad^*(K)\lie{m}^\perp$ for every $f\in H_\rho$.
This implies $\AS_K(\rho)\subset \I\Ad^*(K)\lie{m}^\perp$.
Indeed, if this were not true, we could find a closed cone $S\subset\Weylcc\setminus 0$,
$S\cap\I\Ad^*(K)\lie{m}^\perp=\emptyset$, and
$f=\sum_{\lambda\in\Khat\cap S} f_\lambda\in H_\rho$, $f_\lambda\in W_\lambda$,
such that $\sum_{\lambda\in\Khat\cap S} |\lambda|^{2N} \|f_\lambda\|^2=\infty$ for some $N\in\N$.
Here $W_\lambda$ denotes the $\lambda$-isotypical subspace of $H_\rho$.
Using assumption~\eqref{AS-disjoint-normalbundle} we deduce $\AS_K(\rho)\cap\AS_K(\pi)=\emptyset$.
Therefore $\supp_K(\rho)\cap\supp_K(\pi)$ is relatively compact, hence finite.
By Frobenius reciprocity we get, with sums having only finitely many nonzero summands,
\begin{align*}
m_M(\mu:\pi\restrict{M})
     &= \sum_{\lambda} m_K(\lambda:\pi) m_M(\mu:\lambda\restrict{M}) \\
     &= \sum_{\lambda} m_K(\lambda:\pi) \int_M \overline{\tr\mu(m)} \tr\lambda\restrict{M}(m)\,dm \\
     &= \big(\Theta^K_\pi\restrict{M} \big| \tr\mu\big)_{\Ltwo(M)}.
\end{align*}
The last equation follows from \eqref{K-char-to-M}.

%\bibliographystyle{amsplain}
%\bibliography{abbrev,math}
\providecommand{\bysame}{\leavevmode\hbox to3em{\hrulefill}\thinspace}

\end{document}